\documentclass[11pt,reqno]{amsart}
\usepackage[T1]{fontenc}
\usepackage{lmodern}
\usepackage[letterpaper,margin=1in]{geometry}
\usepackage{amssymb,mathtools,microtype}
\usepackage{xcolor}
\definecolor{linkblue}{RGB}{30,65,100}
\usepackage[colorlinks=true,linkcolor=linkblue,citecolor=linkblue,
  urlcolor=linkblue,breaklinks=true]{hyperref}
\hypersetup{pdftitle={A quasidiagonal C*-algebra with a nonquasidiagonal maximal tensor square},pdfauthor={Mehdi Moradi}}

\newtheorem{theorem}{Theorem}[section]
\newtheorem{proposition}[theorem]{Proposition}
\newtheorem{lemma}[theorem]{Lemma}
\theoremstyle{definition}
\newtheorem{definition}[theorem]{Definition}
\theoremstyle{remark}
\newtheorem{remark}[theorem]{Remark}
\numberwithin{equation}{section}

\newcommand{\C}{\mathbb C}
\newcommand{\Z}{\mathbb Z}
\newcommand{\Q}{\mathbb Q}
\newcommand{\F}{\mathbb F}
\newcommand{\B}{\mathcal B}
\newcommand{\cH}{\mathcal H}
\newcommand{\cK}{\mathcal K}
\newcommand{\cP}{\mathcal P}
\newcommand{\SL}{\mathrm{SL}}
\newcommand{\GL}{\mathrm{GL}}
\newcommand{\tmax}{\otimes_{\max}}
\DeclareMathOperator{\spc}{sp}

\DeclarePairedDelimiter{\norm}{\lVert}{\rVert}

\title[Quasidiagonality and maximal tensor products]{A quasidiagonal $C^*$-algebra with a nonquasidiagonal maximal tensor square}
\author[Mehdi Moradi]{Mehdi Moradi}
\email{mehdi.moradi@utoronto.ca}
\date{Draft, version 3 --- September 16, 2026}

\begin{document}
\begin{abstract}
We construct a separable unital residually finite-dimensional $C^*$-algebra
$A$ whose maximal tensor square contains a proper isometry. In particular,
$A$ is quasidiagonal whereas $A\tmax A$ is not. The algebra is generated by
finite regular representations obtained from congruence quotients of an
explicit countable affine group. A property~$(T)$ subgroup determines a
diagonal Kazhdan projection in the maximal tensor square. Conjugation by a
dilation makes this projection equivalent to a proper subprojection.
Strictness is detected by a quasi-regular representation whose two factor
restrictions descend to the specified algebra $A$. We give the norm estimate
needed for this descent and construct the proper isometry explicitly.
\end{abstract}
\maketitle

\section{Introduction}

This note addresses preservation of quasidiagonality under the maximal
$C^*$-tensor product. The question is explicitly formulated in Brown's survey
\cite[Question~7.6]{BrownSurvey}; the Brown--Ozawa monograph \cite{BO}
develops the approximation theory and the positive free-group example
discussed below. We give an explicit counterexample with two identical
separable unital factors. Each factor has a faithful representation that is a
direct sum of finite-dimensional representations, while its maximal tensor
square contains a proper isometry.

\begin{theorem}\label{thm:main}
There exists a nonzero separable unital residually finite-dimensional
$C^*$-algebra $A$ such that $A$ is quasidiagonal and there is an element
$V\in A\tmax A$ satisfying
\[
V^*V=1,
\qquad
VV^*\ne1.
\]
Consequently, quasidiagonality is not preserved under maximal tensor
products, even for identical separable unital factors.
\end{theorem}

The construction uses a countable affine group $\Gamma$ and a subgroup
$H\cong\Z^3\rtimes\SL_3(\Z)$ with property~$(T)$. There is an element
$t\in\Gamma$ such that
\[
tHt^{-1}\subsetneq H\subsetneq t^{-1}Ht.
\]
The diagonal image of the Kazhdan projection of $H$ gives a projection
$p\in A\tmax A$. With $w=u_t\otimes u_t$, the second subgroup inclusion
implies $w^*pw\le p$. A representation on the cosets of the diagonal
subgroup $\Delta H$ proves that this inequality is strict. The element
$w^*p+(1-p)$ is then the required proper isometry.

The representation on these cosets must be continuous for the particular
completion defining $A$. We establish this by a finite-support norm estimate
and by checking that each factor acts freely on the coset space. Thus both
the projection and the isometry belong to the specified maximal tensor
product.

All groups in this note are discrete, all $C^*$-algebras are complex, and
all Hilbert spaces are complex. We write $\lambda_G$ for the left regular
representation of a group $G$, and $C^*(G)$ and $C^*_r(G)$ for its full and
reduced group $C^*$-algebras. Representations of unital algebras are unital
unless degeneracy is explicitly allowed. The only substantive named group
theorem used below is property~$(T)$ of the above semidirect product,
as stated in \cite[Example~1.7.4(i)]{BHV}.

\subsection{Finite-dimensional approximation and the tensor-product question}

Quasidiagonality is a finite-dimensional approximation property with both
representation-theoretic and completely positive formulations. Voiculescu's
work \cite{Voiculescu} gives its abstract approximation criterion and the
homotopy methods used in several of the examples discussed below. Brown's
survey \cite{BrownSurvey} and the monograph of Brown and Ozawa \cite{BO}
provide systematic accounts of these methods and their permanence
properties. Residual finite dimensionality is a stronger requirement:
finite-dimensional representations themselves must separate points. For a
separable algebra, a faithful direct sum of such representations witnesses
quasidiagonality. This is the particularly concrete form of approximation
enjoyed by the factors in Theorem~\ref{thm:main}.

The minimal tensor product of two quasidiagonal algebras is quasidiagonal
\cite[Proposition~7.5]{BrownSurvey}. Finite-rank approximations to faithful
representations of the factors can be tensored on their Hilbert-space tensor
product. If either factor is nuclear, uniqueness of the tensor norm gives
the same conclusion for the maximal tensor product. Brown explicitly asks
about general maximal tensor products in \cite[Question~7.6]{BrownSurvey}.
The Brown--Ozawa treatment of quasidiagonality, especially its free-group
example \cite[Proposition~7.4.5]{BO}, supplies a central reference point for
this question. Theorem~\ref{thm:main} gives a negative answer to the general
permanence question, with identical RFD factors.

The distinction between the tensor norms is essential here. Tensoring two
faithful factor representations gives a faithful representation of the
minimal completion. For the maximal completion, the same construction
factors through the canonical map to the minimal tensor product and need
not be faithful. Thus that construction does not control all commuting
representations defining the maximal norm. Our proof detects the projection
defect using a particular pair of commuting representations on one Hilbert
space, and verifies their continuity for the chosen completion separately.

\subsection{Free groups and Ozawa's positive result}

Choi proved that the full group algebra $C^*(\mathbb F_2)$ is residually
finite dimensional \cite[Theorem~7]{Choi}; his paper also discusses the
extension to other free ranks. This gives a basic family of noncommutative
full group algebras with faithful finite-dimensional approximation.
Here and below the word \emph{full} matters: the results in this paragraph
are not assertions about reduced free-group algebras.
Exel and Loring \cite[Theorem~3.2]{EL} subsequently established permanence
of residual finite dimensionality under full free products, including
unital full free products with the scalar units identified. Thus RFD
approximation behaves well under these free-product constructions, whereas
the commuting relations imposed in a maximal tensor product present a
different problem.

The corresponding positive maximal-tensor-square result appears in
Brown--Ozawa \cite[Proposition~7.4.5]{BO} and in Ozawa's work on Tsirelson's
problem \cite[Theorem~2]{Ozawa}. In particular, for a free group
$\mathbb F_d$ of finite rank,
\[
 C^*(\mathbb F_d)\tmax C^*(\mathbb F_d)
 \cong C^*(\mathbb F_d\times\mathbb F_d)
 \quad\text{is quasidiagonal}.
\]
Ozawa explicitly credits an idea of Brown and refers to the earlier
monograph argument. His theorem also treats the maximal tensor square of a
free product of finitely many copies of a finite-dimensional commutative
algebra. The free-group proof uses a homotopy of a faithful representation
to a representation with scalar image, followed by the quasidiagonality
homotopy criterion. It does not require equality of the minimal and maximal
tensor norms.

This positive result concerns quasidiagonality, not a claimed separating
family of finite-dimensional representations of the maximal square.
Completely positive approximations that become multiplicative in norm
are different from exact finite-dimensional representations. Maintaining
this distinction is necessary when comparing Ozawa's theorem with the RFD
property of the individual factors constructed here.

\subsection{Atkinson's extension to right-angled Artin groups}

For a finite simplicial graph $\Lambda$, let
\[
 G_\Lambda=\langle s_v\ (v\in V\Lambda):
 [s_v,s_w]=1\ \text{if }\{v,w\}\in E\Lambda\rangle
\]
be its right-angled Artin group. Atkinson proves that $C^*(G_\Lambda)$ is
quasidiagonal \cite[Theorem~4.8]{Atkinson}. The theorem applies to every
finite defining graph; the special graph hypotheses in the tracial-stability
part of that paper are not hypotheses of this quasidiagonality result.
Atkinson explicitly extends the Brown--Ozawa free-group argument.
Generating unitaries in a faithful representation can be deformed to the
identity within their abelian von Neumann algebras. The prescribed
commutation relations persist during the deformation, and
\cite[Proposition~4.7]{Atkinson} then gives quasidiagonality.

To spell out the tensor-product consequence, let $\Lambda*\Omega$ be the
join of two finite graphs, obtained by adding all edges between their
vertex sets. Its defining presentation gives
$G_{\Lambda*\Omega}\cong G_\Lambda\times G_\Omega$: the relations within
each vertex set are retained and all generators from different sets commute.
The universal property of full group algebras therefore identifies
\[
 C^*(G_\Lambda)\tmax C^*(G_\Omega)
 \cong C^*(G_{\Lambda*\Omega}).
\]
Applying Atkinson's theorem to the join proves quasidiagonality of this
maximal product, including each tensor square. Taking edgeless graphs
recovers the free-group examples. This join argument is the stated
consequence of his theorem, rather than an additional hypothesis on the
graphs. Caspers and Fima \cite[Definition~2.4 and Example~2.6]{CF} supply the
broader maximal graph-product framework and its compatibility with full
group algebras. These positive cases show that the maximal tensor norm alone
is not an obstruction; the additional subgroup geometry used below is
decisive for this construction.

\subsection{Kazhdan projections, quotient obstructions, and this construction}

Wassermann's work \cite{WassermannT,WassermannQuotient} is an important
antecedent to the use of finite-dimensional group representations and
Kazhdan property~$(T)$ in failures of permanence. In particular, he
constructed a separable quasidiagonal operator algebra whose quotient by
the compact operators is not quasidiagonal. Brown discusses these examples
in \cite[Example~8.6]{BrownSurvey}. They show why quasidiagonality of a
faithfully represented block algebra cannot simply be passed to a quotient.
They also provide historical context for the spectral-gap projections
arising from property~$(T)$.

The conclusion needed here is specifically an obstruction inside the
maximal tensor product itself. Property~$(T)$ of
$\Z^3\rtimes\SL_3(\Z)$ supplies the invariant-vector projection, using
\cite[Example~1.7.4(i)]{BHV}. The strict conjugate-subgroup inclusion then
produces two equivalent projections with one properly below the other.
The finite-quotient norm estimate ensures that the coset representation
detecting strictness is a representation of the actual algebra
$A\tmax A$. Consequently the resulting proper isometry belongs to that
algebra, rather than merely to a quotient or a larger operator algebra.
This identifies the precise feature separating the construction from
quotient-based failures and from the positive free-group and Artin-group
results above. No identification of $A$ with the full group algebra
$C^*(\Gamma)$, and no permanence under quotients, is assumed.

\section{Quasidiagonality and proper isometries}

\begin{definition}
A separable $C^*$-algebra $B$ is \emph{quasidiagonal} if it admits a faithful
representation $\pi:B\to\B(\cH)$ on a separable Hilbert space and an
increasing sequence $(R_n)$ of finite-rank orthogonal projections such that
$R_n\to I_{\cH}$ strongly and
\[
\norm{[R_n,\pi(b)]}\longrightarrow0
\quad\text{for every }b\in B.
\]
The representation in this definition is allowed to be degenerate.
A $C^*$-algebra is \emph{residually finite dimensional} if its
finite-dimensional $*$-representations separate its points.
\end{definition}

We record the finiteness obstruction in a form that also covers a degenerate
faithful representation.

\begin{lemma}\label{lem:finite}
If $B$ is a unital quasidiagonal $C^*$-algebra and $s\in B$ satisfies
$s^*s=1$, then $ss^*=1$.
\end{lemma}

\begin{proof}
Choose a faithful quasidiagonal representation $\pi:B\to\B(\cH)$ with
projections $(R_n)$ as above. Put $e=\pi(1)$ and restrict the representation
to $e\cH$, where it is faithful and unital. The finite-rank positive
contractions $T_n=eR_ne$ on $e\cH$ converge strongly to its identity and
satisfy
\[
T_n-T_n^2=eR_n(1-e)R_ne,
\qquad
\norm{T_n-T_n^2}\le\norm{[R_n,e]}^2\longrightarrow0.
\]
For all sufficiently large $n$, the spectrum of $T_n$ avoids $1/2$.
Let $Q_n=\mathbf1_{[1/2,1]}(T_n)$. These are finite-rank projections, and
\[
\norm{Q_n-T_n}\le 2\norm{T_n-T_n^2}\longrightarrow0.
\]
They therefore converge strongly to the identity on $e\cH$. Moreover,
for $b\in B$,
\[
[T_n,\pi(b)|_{e\cH}]
=e[R_n,\pi(b)]e|_{e\cH}.
\]
It follows that $Q_n$ asymptotically commutes with the restricted
representation. Increasingness of $(Q_n)$ is not needed below.

Write $S=\pi(s)|_{e\cH}$ and $\delta_n=\norm{[Q_n,S]}$. On the
finite-dimensional space $Q_ne\cH$, put $X_n=Q_nSQ_n$. Since $S^*S=1$,
\[
Q_n-X_n^*X_n=Q_nS^*(1-Q_n)SQ_n,
\qquad
\norm{Q_n-X_n^*X_n}\le\delta_n^2.
\]
The two positive matrices $X_n^*X_n$ and $X_nX_n^*$ have the same
eigenvalues, counted with multiplicity. Hence
\[
\norm{Q_n-X_nX_n^*}=\norm{Q_n-X_n^*X_n}\le\delta_n^2.
\]
Also,
\[
Q_nSS^*Q_n=X_nX_n^*+Q_nS(1-Q_n)S^*Q_n,
\]
and the second summand has norm at most $\delta_n^2$. Consequently,
\[
\norm{Q_n(1-SS^*)Q_n}\le2\delta_n^2\longrightarrow0.
\]
Strong convergence of $Q_n$ to the identity forces $1-SS^*=0$.
Faithfulness then gives $ss^*=1$ in $B$.
\end{proof}

\section{The group and its finite-dimensional completion}

Let $\Z[1/2]=\{m/2^j:m\in\Z,\ j\ge0\}$ and define
\begin{equation}\label{eq:group}
\Gamma=
\left\{
\begin{pmatrix}2^kM&b\\0&1\end{pmatrix}:
k\in\Z,\ M\in\SL_3(\Z),\ b\in\Z[1/2]^3
\right\}\subset\GL_4(\Q).
\end{equation}
For calculations, write $(L,b)$ for the affine matrix with linear part $L$
and translation part $b$. Then
\[
(L,b)(N,c)=(LN,b+Lc),
\qquad
(L,b)^{-1}=(L^{-1},-L^{-1}b).
\]
These formulas show that \eqref{eq:group} is a group. The determinant of
its linear part is $2^{3k}$, so the displayed parametrization is injective.
In particular, $\Gamma$ is countable.

Let $\cP$ be the set of odd primes. For $\ell\in\cP$, reduction of matrix
entries gives a homomorphism
\[
r_\ell:\Gamma\longrightarrow\GL_4(\F_\ell).
\]
It is well defined because $2$ is invertible modulo $\ell$, and both each
matrix and its inverse have entries in $\Z[1/2]$. Put
\[
G_\ell=r_\ell(\Gamma),
\qquad
\pi_\ell=\lambda_{G_\ell}\circ r_\ell.
\]

\begin{lemma}\label{lem:separation}
For every finite subset $E\subset\Gamma$, there is an odd prime $\ell$
such that $r_\ell$ is injective on $E$.
\end{lemma}

\begin{proof}
For each pair of distinct matrices in $E$, choose a nonzero entry of their
difference. Write each chosen entry as $m/2^j$ with $m\in\Z\setminus\{0\}$.
An odd prime dividing none of these finitely many numerators exists. Every
chosen difference remains nonzero modulo that prime, so all pairs remain
distinct. The assertion is immediate if $E$ has at most one element.
\end{proof}

On the separable Hilbert space
\[
\cK=\bigoplus_{\ell\in\cP}\ell^2(G_\ell)
\]
define
\begin{equation}\label{eq:algebra}
u_g=\bigoplus_{\ell\in\cP}\pi_\ell(g),
\qquad
A=C^*(u_g:g\in\Gamma)\subset\B(\cK).
\end{equation}

\begin{proposition}\label{prop:factor}
The algebra $A$ is nonzero, separable, unital, residually finite dimensional,
and quasidiagonal. There is a unital surjective $*$-homomorphism
\begin{equation}\label{eq:regular-map}
\theta:A\longrightarrow C^*_r(\Gamma),
\qquad
\theta(u_g)=\lambda_\Gamma(g).
\end{equation}
\end{proposition}

\begin{proof}
Since $\Gamma$ is countable, $A$ is separable, and $u_1=I_{\cK}\ne0$.
The concrete inclusion in \eqref{eq:algebra} is faithful. Enumerate the odd
primes and let $P_N$ project onto the first $N$ whole summands of $\cK$.
These projections have finite rank, increase strongly to $I_{\cK}$, and
commute with every generator $u_g$, hence with every element of $A$.
Thus this representation witnesses quasidiagonality. Restriction to each
finite-dimensional summand is a $*$-representation of $A$, and these
restrictions separate its points, proving residual finite dimensionality.

To prove \eqref{eq:regular-map}, let
$x=\sum_{s\in S}c_s s\in\C[\Gamma]$ have finite support and let
$\eta\in\ell^2(\Gamma)$ have finite support $F$. Choose $\ell$ so that
$r_\ell$ is injective on $F\cup SF$. The basis assignment
$\delta_g\mapsto\delta_{r_\ell(g)}$ is then an isometry on the span of
these basis vectors. Denoting it by $J$, expansion of the two finite sums
gives
\[
J\lambda_\Gamma(x)\eta=\pi_\ell(x)J\eta.
\]
In particular,
\begin{align*}
\norm{\lambda_\Gamma(x)\eta}
&\le\norm{\pi_\ell(x)}\norm{\eta}\\
&\le\sup_{m\in\cP}\norm{\pi_m(x)}\norm{\eta}
=\norm{u(x)}\norm{\eta}.
\end{align*}
Taking the supremum over finite-support unit vectors yields
\begin{equation}\label{eq:regular-bound}
\norm{\lambda_\Gamma(x)}\le\norm{u(x)}
\qquad(x\in\C[\Gamma]).
\end{equation}
Thus the assignment $u(x)\mapsto\lambda_\Gamma(x)$ is well defined and
contractive on the dense group-algebra image in $A$. It extends to a unital
$*$-homomorphism $\theta$. Its image is closed and contains the regular
generators, so it is all of $C^*_r(\Gamma)$.
\end{proof}

\section{The subgroup and its invariant-vector projection}

Inside $\Gamma$, set
\begin{equation}\label{eq:subgroup}
H=\left\{
\begin{pmatrix}M&b\\0&1\end{pmatrix}:
M\in\SL_3(\Z),\ b\in\Z^3
\right\},
\qquad
t=\begin{pmatrix}2I_3&0\\0&1\end{pmatrix}.
\end{equation}
The group $H$ is $\Z^3\rtimes\SL_3(\Z)$ for the usual linear action.
It has ordinary Kazhdan property~$(T)$ by
\cite[Example~1.7.4(i), p.~68]{BHV}. That source writes the acting factor
first, as $\SL_3(\Z)\ltimes\Z^3$.

Direct multiplication gives
\begin{align}
tHt^{-1}
&=\{(M,2b):M\in\SL_3(\Z),\ b\in\Z^3\}
\subsetneq H,\label{eq:contracted}\\
t^{-1}Ht
&=\{(M,b/2):M\in\SL_3(\Z),\ b\in\Z^3\}
\supsetneq H.\label{eq:expanded}
\end{align}
For example, translation by the first coordinate vector $e_1$ belongs to
$H\setminus tHt^{-1}$, and translation by $e_1/2$ belongs to
$t^{-1}Ht\setminus H$.

We next construct the projection supplied by property~$(T)$.

\begin{lemma}\label{lem:kazhdan}
There is a projection $e_H\in C^*(H)$ such that, for every unitary
representation $U$ of $H$, its integrated representation maps $e_H$ to the
orthogonal projection onto the $H$-invariant vectors.
\end{lemma}

\begin{proof}
Choose a finite Kazhdan set $S_H\subset H$ and a constant $\varepsilon>0$.
In a unitary representation $U$, split off the invariant-vector subspace.
Its orthogonal complement is invariant and contains no nonzero invariant
vectors. On this complement the Kazhdan gap implies
\begin{equation}\label{eq:gap}
\sum_{s\in S_H}\norm{U_s\xi-\xi}^2
\ge\varepsilon^2\norm{\xi}^2.
\end{equation}
Writing $z_h$ for the canonical unitaries of $C^*(H)$, form
\[
d=\sum_{s\in S_H}(1-z_s)^*(1-z_s).
\]
In a faithful universal representation, \eqref{eq:gap} shows that
\[
\spc(d)\subset\{0\}\cup[\varepsilon^2,4|S_H|].
\]
Choose a continuous function $f$ with $f(0)=1$ and
$f|_{[\varepsilon^2,\infty)}=0$, and put $e_H=f(d)$. This is a projection.
The same gap in every representation shows that its image is exactly the
invariant-vector projection. The trivial representation shows $e_H\ne0$.
\end{proof}

Put $D=A\tmax A$. The diagonal unitary representation
$h\mapsto u_h\otimes u_h$ induces a unital $*$-homomorphism
\[
j:C^*(H)\longrightarrow D.
\]
Define
\begin{equation}\label{eq:projections}
p=j(e_H),\qquad w=u_t\otimes u_t,\qquad q=w^*pw.
\end{equation}

\begin{lemma}\label{lem:order}
The projections in \eqref{eq:projections} satisfy $q\le p$ in $D$.
\end{lemma}

\begin{proof}
In any unital representation $\rho:D\to\B(\cH)$, Lemma~\ref{lem:kazhdan}
identifies $\rho(p)$ with the projection onto vectors fixed by
$\rho(u_h\otimes u_h)$ for every $h\in H$. If $W=\rho(w)$, the range of
$\rho(q)=W^*\rho(p)W$ consists of those $\xi$ for which $W\xi$ is fixed by
the diagonal $H$ action. Equivalently, $\xi$ is fixed by
\[
\rho(u_{t^{-1}ht}\otimes u_{t^{-1}ht})
\quad\text{for every }h\in H.
\]
Thus $\rho(q)$ projects onto diagonal-$t^{-1}Ht$-invariant vectors.
By \eqref{eq:expanded}, this subspace is contained in the range of
$\rho(p)$. Taking $\rho$ faithful proves $q\le p$ in $D$.
\end{proof}

\section{A representation detecting strictness}

Let
\[
L=\Delta H=\{(h,h):h\in H\},
\qquad
X=(\Gamma\times\Gamma)/L,
\]
where $X$ is the set of left cosets. Write $\sigma$ for the quasi-regular
representation on $\ell^2(X)$:
\[
\sigma(a,b)\delta_{gL}=\delta_{(a,b)gL}.
\]

\begin{proposition}\label{prop:cosets}
There is a unital representation $\Sigma:D\to\B(\ell^2(X))$ such that
\[
\Sigma(u_a\otimes u_b)=\sigma(a,b)
\qquad(a,b\in\Gamma).
\]
\end{proposition}

\begin{proof}
Each individual factor acts freely on $X$. Indeed, if $(a,1)$ stabilizes
$(g_1,g_2)L$, then
\[
(g_1^{-1}ag_1,1)\in\Delta H,
\]
which forces $a=1$. For the second factor the corresponding condition is
$(1,g_2^{-1}bg_2)\in\Delta H$, which forces $b=1$.
Every orbit of either restricted action is therefore equivariantly a copy
of $\Gamma$ with its left translation action. Each restricted unitary
representation is a direct sum of copies of $\lambda_\Gamma$.

By Proposition~\ref{prop:factor}, both restrictions extend to
representations $\alpha,\beta:A\to\B(\ell^2(X))$ on this same Hilbert
space, with
\[
\alpha(u_a)=\sigma(a,1),\qquad
\beta(u_b)=\sigma(1,b).
\]
Their ranges commute, since their generating group actions commute and
commutation is preserved under norm limits. The universal property of the
maximal tensor product now gives $\Sigma$, with
$\Sigma(a\otimes b)=\alpha(a)\beta(b)$.
\end{proof}

\begin{proposition}\label{prop:strict}
The inequality in Lemma~\ref{lem:order} is strict: $q<p$ in $D$.
\end{proposition}

\begin{proof}
Let $\xi=\delta_L\in\ell^2(X)$. It is fixed by every element of $\Delta H$,
so $\Sigma(p)\xi=\xi$. Let $h_0=(I_3,e_1)\in H$ and set
\[
k=t^{-1}h_0t=(I_3,e_1/2)\in t^{-1}Ht\setminus H.
\]
Since $(k,k)\notin L$, the cosets $(k,k)L$ and $L$ differ. Hence
\[
\sigma(k,k)\xi=\delta_{(k,k)L}\ne\delta_L.
\]
Thus $\xi$ is not invariant under the diagonal action of $t^{-1}Ht$.
By the invariant-subspace description in the proof of
Lemma~\ref{lem:order}, $\Sigma(q)\xi\ne\xi$. Therefore
$\Sigma(p-q)\xi\ne0$, which implies $p-q\ne0$ in $D$.
Together with $q\le p$, this proves the assertion.
\end{proof}

\section{The proper isometry}

\begin{proof}[Proof of Theorem~\ref{thm:main}]
Use the exact algebra $A$ of \eqref{eq:algebra}. Its asserted factor
properties were proved in Proposition~\ref{prop:factor}. With $p,w,q$ as
in \eqref{eq:projections}, let $v=w^*p$. Then
\[
v^*v=p,\qquad vv^*=q<p.
\]
Moreover, $vp=v$ and $qv=v$; since $pq=q$, it follows that $pv=v$.
Thus $v\in pDp$. Set
\[
V=v+(1-p)=w^*p+(1-p).
\]
All mixed terms with $1-p$ vanish in both products. Consequently,
\[
V^*V=p+(1-p)=1,
\qquad
VV^*=q+(1-p)=1-(p-q)\ne1.
\]
This is a proper isometry in $D$ itself. Lemma~\ref{lem:finite} therefore
excludes quasidiagonality of $D$, completing the proof.
\end{proof}

\begin{remark}
The representation $\Sigma$ is used only to detect that $p-q$ is nonzero;
it need not be faithful. The order relation $q\le p$ is established
separately in a faithful representation of $D$. Both steps are needed to
obtain the proper isometry in $D$.
\end{remark}

\begin{remark}
The counterexample uses nonzero unital factors, and hence also refutes the
universal preservation statement when zero and nonunital separable
algebras are allowed. No permanence of quasidiagonality under quotients is
used: the map $\theta$ serves only to construct the two commuting
representations in Proposition~\ref{prop:cosets}.
\end{remark}

\subsection*{Acknowledgements}
The author acknowledges the substantial assistance of OpenAI GPT-family
models in exploring possible approaches, developing the counterexample,
checking its individual steps in separate model contexts, and preparing
the exposition and the \LaTeX{} source. The author also notes the recent
appearance of similar AI-assisted results and acknowledges the possibility
of overlap with arguments generated independently using AI. No specific
claim of equivalence or priority is made here.

\end{document}